\documentclass[letterpaper]{article}

\usepackage{amsmath}
\usepackage{amsfonts}
\usepackage{amsthm}
\usepackage{amssymb}
\usepackage{datetime}
\usepackage{color}
\usepackage{url}

\newtheorem{theorem}{Theorem}[section]
\newtheorem*{theorem*}{Theorem}
\newtheorem{lemma}[theorem]{Lemma}

\theoremstyle{definition}
\newtheorem{definition}[theorem]{Definition}
\newtheorem{problem}[theorem]{Problem}

\newtheorem{remark}[theorem]{Remark}

\numberwithin{equation}{section}

\newcommand{\all}{\hbox{for all}}
\newcommand{\bra}[2]{\langle#1,#2\rangle}
\newcommand{\Bra}[2]{\big\langle#1,#2\big\rangle}
\newcommand{\dbs}{^{**}}
\newcommand{\eps}{\varepsilon}
\newcommand{\exs}{\hbox{there exists}}

\newcommand{\fourth}{{\textstyle\frac{1}{4}}}
\newcommand{\lr}{\Longrightarrow}
\newcommand{\NN}{\mathbb N}

\newcommand{\RR}{\mathbb R}
\newcommand{\st}{\hbox{such that}}

\newcommand{\toto}{\rightrightarrows}
\newcommand{\trs}{^{***}}
\newcommand{\ts}{\textstyle}
\newcommand{\wh}{\widehat}

\newcommand{\Lem}{Lemma~\ref}

\newcommand{\Rem}{Remark~\ref}
\newcommand{\Rems}{Remarks~\ref}
\newcommand{\Sec}{Section~\ref}

\newcommand{\Thm}{Theorem~\ref}

\title{Gossez's skew linear map and its pathological maximally monotone multifunctions}
\author{
Stephen Simons
\thanks{
Department of Mathematics, University of California, Santa Barbara, CA\ 93106-3080, U.S.A.
Email: \texttt{stesim38@gmail.com}.}}
\date{}

\begin{document}

\maketitle

\begin{abstract}\noindent
In this note, we give a generalization of Gossez's example of a maximally monotone multifunction such that the closure of its range is not convex, using more elementary techniques  than in Gossez's original papers.   We also discuss some new properties of Gossez's skew linear operator and its adjoint.
\end{abstract}

{\small \noindent {\bfseries 2010 Mathematics Subject Classification:}
{Primary 47H05; Secondary 47N10, 46A20, 46A22.}}

\noindent {\bfseries Keywords:} Skew linear operator, maximal monotonicity, duality map.

\section{Introduction}
In \cite{GOSSEZRANGE} and \cite{GOSSEZCONV}, Gossez gives an example of a skew linear map $G\colon\ \ell_1 \to \ell_\infty = \ell_1^*$, and proves that there exist arbitrarily small values of $\lambda > 0$ such that $\overline{R(G + \lambda J)}$ is not convex.   \big(If $E$ is a Banach space, $J\colon\ E \toto E^*$ is the {\em duality map}, defined by {\em $x^* \in Jx$ exactly when $\|x^*\| =  \|x\|$ and $\bra{x}{x^*} = \|x\|^2$}.   See eqn. \eqref{A3}.\big)
   In \Thm{PATHOthm}, we shall prove the stronger result that $\overline{R(G + \lambda J)}$ is not convex whenever $0  < \lambda < 4$.   In particular, $\overline{R(G + J)}$ is not convex.   \big($R(\cdot)$ stands for ``range of''.\big)
\par
Gossez's analysis goes by way of the {\em monotone extension to the dual}\break introduced in \cite{GOSSEZ}.   This was critical to his definition of operators of {\em dense type}, which have been so important in the modern theory of monotone\break multifunctions.   In addition to the use of the monotone extension to the dual,  \cite{GOSSEZRANGE} and \cite{GOSSEZCONV} use measure theory on the Stone-\u Cech compactification of the positive integers.   In this paper, we use mainly elementary functional analysis, but we will make some comments about the measure theoretic approach in \Rems{BETArem} and \ref{GTMrem}.
\par
In \Sec{GENsec}, we define a {\em skew} linear operator, $A$, from a Banach space, $E$ into its dual and, in \Thm{Athm}, we establish an upper bound for the quadratic form $-\bra{A^* x\dbs}{x\dbs}$ on $E\dbs \times E\dbs$ under certain circumstances.   See eqn.\ \eqref{A2}.
\par
In eqn.\eqref{GDEF}, we give the exact formula for $G$.    Our presentation exploits the fact that $G$ can be ``factorized through $c$''.   In \Lem{GTlem}, we discuss a particular element $x_0\dbs \in \ell_1\dbs = \ell_\infty^*$ and give formulae for $G^*x_0\dbs$ and $\bra{G^*x_0\dbs}{x_0\dbs}$.   \Lem{PROPlem} appears in \cite[Proposition, p.\ 360]{GOSSEZCONV}, but with a very different proof.   \Lem{PROPlem} leads rapidly to our main result, \Thm{PATHOthm}.
\par
In \Sec{TRIsec}, we give some technical results on $\ell_1$, $\ell_1^*$, $\ell_1\dbs$ and $\ell_1\trs$ and, in \Thm{Wthm}, define a particular element $w\trs$ of $\ell_1\trs$ that will be used in \Sec{FINERsec} to obtain formulae for $G^*x\dbs$ and $\bra{G^*x\dbs}{x\dbs}$ for {\em general} $x\dbs \in \ell_1\dbs$.   It was proved in \cite[Example, p.\ 89]{GOSSEZRANGE} and  \cite[Example 14.2.2, pp.\ 161--162]{HEINZ} that, for all $x\dbs \in \ell_1\dbs$, $\bra{G^*x\dbs}{x\dbs} \le 0$.   In \eqref{TRZ1}, we strengthen these results by showing that $\bra{G^* x\dbs}{x\dbs} = -{\bra{x\dbs}{w\trs}}^2$.   
\par 
All Banach spaces considered in this note are {\em real}.
\par
The author would like to express his thanks to Heinz Bauschke for reading the first draft of this paper and making a number of extremely valuable suggestions.   He would also like to thank Jerry Beer for a very illuminating discussion on compactifications, which simplified considerably the analysis in \Rem{BETArem}.
%:  \Sec{GENsec}
\section{On skew linear operators on general Banach spaces}\label{GENsec}  
%:   \Def{SKEWdef}
\begin{definition}\label{SKEWdef}
Let $E$ be a nonzero Banach space and $A\colon\ E \to E^*$ be linear.   We say that $A$ is {\em skew} if,
%   \eqref{SKEW1}
\begin{equation}\label{SKEW1}
\all\ w,x \in E,\ \bra{w}{Ax} = -\bra{x}{Aw},
\end{equation}
or, equivalently,   
%   \eqref{SKEW2}
\begin{equation}\label{SKEW2}
\all\ x \in E,\ \bra{x}{Ax} = 0.
\end{equation}
If $x \in E$, we write $\wh x$ for the canonical image of $x$ in $E\dbs$, that is to say $x \in E$ and $x^* \in E^* \lr \bra{x^*}{\wh x} = \bra{x}{x^*}$.
\end{definition}
\par
We recall that if $X$ and $Y$ are Banach spaces and $A\colon\ X \to Y$ is linear then the adjoint $A^*\colon\ Y^* \to X^*$ is defined by $\bra{x}{A^*y^*} = \bra{Ax}{y^*}$\quad($x \in X$, $y^* \in Y^*$).
%:   \Thm{Athm}
\begin{theorem}\label{Athm}
Let $A\colon\ E \to E^*$ be bounded, skew and linear.   Suppose that $x\dbs \in E\dbs$, $\lambda > 0$ and $A^* x\dbs \in \overline{R(A + \lambda J)}$.   Then
%   \eqref{A2}
\begin{equation}\label{A2}
-\bra{A^* x\dbs}{x\dbs} \le \fourth\lambda\|x\dbs\|^2.
\end{equation}
\end{theorem}    
\begin{proof}
Let $\eps > 0$.  By hypothesis, there exist $x \in E$, $x^* \in E^*$ with
%   \eqref{A3}
\begin{equation}\label{A3}
\|x^*\| =  \|x\|,\ \bra{x}{x^*} = \|x\|^2,
\end{equation} 
and $z^* \in E^*$ such that $\|z^*\| < \eps$ and $A^* x\dbs = Ax + \lambda x^* + z^*$.  Then   
%   \eqref{A6}
\begin{equation}\label{A6}
A^* x\dbs - Ax = \lambda x^* + z^*,
\end{equation} 
and, using \eqref{A3},
%   \eqref{A4}
\begin{equation}\label{A4}
-\bra{x^*}{x\dbs} \le \|x^*\|\|x\dbs\| = \|x\dbs\|\|x\|.
\end{equation}
Let $Z_\eps := \|x\dbs\| + \eps/\lambda$.  From the definition of $A^*$, \eqref{SKEW2}, \eqref{A3}--\eqref{A4}, and the inequalities $\|z^*\| < \eps$ and  $-\|x\|^2 + Z_\eps\|x\| \le \fourth Z_\eps ^2$,  
\begin{align*}
-\bra{A^* x\dbs}{x\dbs}
&= -\bra{x}{A^* x\dbs - Ax} - \bra{A^* x\dbs - Ax}{x\dbs}\\
&= -\bra{x}{\lambda x^* + z^*} - \bra{\lambda x^* + z^*}{x\dbs}\\
&\le -\lambda\|x\|^2 + \eps\|x\| + \lambda \|x\dbs\|\|x\| + \eps\|x\dbs\|\\
&= -\lambda\|x\|^2 + \lambda Z_\eps\|x\| + \eps\|x\dbs\|\le \fourth\lambda Z_\eps^2 + \eps\|x\dbs\|.
\end{align*}
Since $Z_\eps \to \|x\dbs\|$ as $\eps \to 0$, \eqref{A2} now follows by letting $\eps \to 0$.
\end{proof}
%:   \Sec{GOSSsec}
\section{Gossez's skew linear operator}\label{GOSSsec}
\begin{definition}\label{Hdef}
In the interest of precision, we shall use three different notations for the three duality pairings that appear in the rest of this paper.   Then (noting that  $\ell_1^* = \ell_\infty$), the bilinear form $\bra{\cdot}{\cdot}_0\colon\ \ell_1 \times \ell_1^* \to \RR$ is\break defined in the usual way.      Then $\ell_1\dbs = \ell_\infty^*$, but this space does not have a convenient sequential representation.   In this connection, see  \Rem{BETArem}.   Also, $\bra{\cdot}{\cdot}_1\colon\ \ell_1^* \times \ell_1\dbs \to \RR$ and $\bra{\cdot}{\cdot}_2\colon\ \ell_1\dbs \times \ell_1\trs \to \RR$.   We write $\|\cdot\|_1$ and $\|\cdot\|_\infty$ for the usual norms on $\ell_1$ and $\ell_\infty$.  Let $c$ be the subspace of $\ell_\infty$ consisting of all {\em convergent} sequences.   Finally, let $e^* := (1,1,\dots) \in c$.   In what follows, all sequences are indexed by the set $\{1,2,3,\dots\}$.
\par
Define the linear operator $G\colon\ \ell_1 \to \ell_1^* = \ell_\infty$ by
%   \eqref{GDEF}
\begin{equation}\label{GDEF}
\all\ x \in \ell_1,\ (Gx)_m = \ts\sum_{n > m} x_n - \sum_{n < m}x_n.
\end{equation}
$G$ is the ``Gossez operator''.  It is well known that $G$ is skew and maximally monotone.  See \cite[Example, p.\ 89]{GOSSEZRANGE}.    Clearly, for all $x \in \ell_1$,
%   \eqref{H1}
\begin{equation}\label{H1}
\ts\lim_{m \to \infty}(Gx)_m = - \sum_{n = 1}^\infty x_n = -\bra{x}{e^*}_0.
\end{equation}
\end{definition}
%:   \Lem{GTlem}
\begin{lemma}\label{GTlem}
There exists $x_0\dbs \in \ell_1\dbs = \ell_\infty^*$ such that
%   \eqref{GT0}
\begin{equation}\label{GT0}
\|x_0\dbs\| = 1,
\end{equation}
%   \eqref{GT1}
\begin{equation}\label{GT1}
G^* x_0\dbs = - e^* \in \ell_1^* =  \ell_\infty
\end{equation}
and
%   \eqref{GT2}
\begin{equation}\label{GT2}
\bra{G^* x_0\dbs}{x_0\dbs}_1 = - 1.
\end{equation}
\end{lemma}
\begin{proof}
The map from $c$ into $\RR$ defined by $x^* \mapsto \lim_{m \to \infty} x_m^*$ is bounded and linear and has norm $1$ on the vector subspace $c$ of $\ell_\infty$.   So, from the extension form of the Hahn--Banach theorem, there exists $x_0\dbs \in \ell_\infty^*$ such that \eqref{GT0} is satisfied and,
%   \eqref{R1}
\begin{equation}\label{R1}
\ts\all\ x^* \in c,\ \bra{x^*}{x_0\dbs}_1 = \lim_{m \to \infty} x_m^*.
\end{equation}        
For all $x \in \ell_1$, $Gx \in c$.   Thus, from  \eqref{H1},
%   \eqref{GT3}
\begin{equation}\label{GT3}
\bra{x}{G^* x_0\dbs}_0 = \bra{Gx}{x_0\dbs}_1 = \ts\lim_{m \to \infty}(Gx)_m = - \bra{x}{e^*}_0.
\end{equation}
This completes the proof of \eqref{GT1}.   From this, $\bra{G^* x_0\dbs}{x_0\dbs}_1 = \bra{- e^*}{x_0\dbs}_1$ and
\eqref{GT2} is immediate from \eqref{R1}.
\end{proof}
\par
The proof of \Lem{PROPlem} below is based on that of \cite[Proposition, p.\ 360]{GOSSEZCONV}.   However, instead of using measure theory on $\beta \NN$, we use the fact that a linear subspace is closed under differences (in \eqref{PROP5} and \eqref{PROP6}) and sums (in \eqref{PROP7}).   There is another way of establishing \Lem{PROPlem}, using {\em Rugged Banach spaces}.   See \cite[Proposition 15.3.8, p.\ 176]{HEINZ}. 
%:   \Lem{PROPlem} 
\begin{lemma}\label{PROPlem}
Let $\lambda > 0$.   Suppose that
\begin{equation}\label{PROP1}
\overline{R(G + \lambda J)}\hbox{ is convex.}
\end{equation}
Then
%   \eqref{PROP2}}
\begin{equation}\label{PROP2}
\overline{R(G + \lambda J)} = \ell_\infty.
\end{equation}
\end{lemma}
\begin{proof}
Let $k \ge 1$.   As observed in \cite[Proposition, p.\ 360]{GOSSEZCONV}, if, for all $m \not\in \{1,2\}$, $|u_m^*| \le 2\lambda k$, then
%   \eqref{PROP3}
\begin{equation}\label{PROP3}
(-k + 2\lambda k,-k - 2\lambda k,u_3^*,u_4^*,u_5^*,\dots) \in (G + \lambda J)(ke_1 - ke_2) \in \overline{R(G + \lambda J)}.
\end{equation}   
In particular,
%   \eqref{PROP4}
\begin{equation}\label{PROP4}
(-k + 2\lambda k,-k - 2\lambda k,0,0,0,\dots) \in \overline{R(G + \lambda J)}.
\end{equation}   
As observed in \cite[Proposition, p.\ 360]{GOSSEZCONV}, \eqref{PROP1} implies that $\overline{R(G + \lambda J)}$ is a linear subspace of $\ell_\infty$.   So, by subtracting \eqref{PROP4} from \eqref{PROP3},
%   \eqref{PROP5}
\begin{equation}\label{PROP5}
(0,0,u_3^*,u_4^*,u_5^*,\dots) \in \overline{R(G + \lambda J)}.
\end{equation}   
Similarly, if, for all $m \not\in \{3,4\}$, $|v_m^*| \le 2\lambda k$, then
%   \eqref{PROP6}
\begin{equation}\label{PROP6}
(v_1^*,v_2^*,0,0,v_5^*,\dots) \in \overline{R(G + \lambda J)}.
\end{equation}   
Taking the Minkowski sum of \eqref{PROP5} and \eqref{PROP6},
%   \eqref{PROP7}
\begin{equation}\label{PROP7}
(v_1^*,v_2^*,u_3^*,u_4^*,u_5^* + v_5^*,u_6^* + v_6^*,\dots) \in \overline{R(G + \lambda J)}.
\end{equation}   
\eqref{PROP2} now follows easily by letting $k \to \infty$.
\end{proof}
%:   \Lem{AGthm}
\begin{lemma}\label{AGthm}
Suppose that $\lambda > 0$ and $\overline{R(G + \lambda J)}$ is convex.  Then $\lambda \ge 4$.
\end{lemma}    
\begin{proof}
Let $x_0\dbs$ be as in \Lem{GTlem}.   From \Lem{PROPlem}, $G^* x_0\dbs \in \overline{R(G + \lambda J)}$.  From \Thm{Athm}, \eqref{GT2} and \eqref{GT0}, $1 = -\bra{G^* x_0\dbs}{x_0\dbs} \le \fourth\lambda\|x_0\dbs\|^2 = \fourth\lambda$. This gives the desired result.
\end{proof}
%:   \Thm{PATHOthm}
\begin{theorem}\label{PATHOthm}
If $0 < \lambda < 4$ then $\overline{R(G + \lambda J)}$ is not convex.    In particular, $\overline{R(G + J)}$ is not convex.
\end{theorem} 
\begin{proof}
This is immediate from \Lem{AGthm}.
\end{proof}
\begin{problem}
Is $\overline{R(G + 4J)}$ convex?
\end{problem}
%:   \Sec{TRIsec}
\section{On the dual, bidual and tridual of $\ell_1$}\label{TRIsec}
This section is devoted to the technical results that will be needed for our discussion of $G^*$ in \Sec{FINERsec}.   We point, in particular, to \Lem{Wlem}(c), in which $p^*$ is moved from being the first variable in $\bra{\cdot}{\cdot}_1$ to being the second variable in $\bra{\cdot}{\cdot}_0$, {\em i.e.}, from being a primal variable to being a dual variable.   \Lem{Wlem}(c) will be critical in the proof of \eqref{W3}, which will be used in \Thm{TRZthm}.
\par
Let $c_0$ be the Banach space of sequences that converge to $0$.   For all $m \ge 1$, let $e_m^*$ be the element $(0,\dots, 0,1,0,0,\dots)$ of $\ell_1^*$, with the $1$ in the $m$th place.   Define the linear map $W\colon\ \ell_1\dbs \to \ell_1$ by $Wx\dbs := \big(\bra{e_m^*}{x\dbs}_1\big)_{m \ge 1}$.    
%:   \Lem{Wlem}
\begin{lemma}\label{Wlem}
\par\noindent
{\rm(a)}\enspace Let $x\dbs \in \ell_1\dbs$.   Then $\sum_{m = 1}^\infty|\bra{e_m^*}{x\dbs}_1| \le \|x\dbs\| < \infty$.
\par\noindent
{\rm(b)}\enspace $\|W\| = 1$ and, for all $x \in \ell_1$, $W{\wh x} = x$.
\par\noindent
{\rm(c)}\enspace Let $p^* \in c_0 \subset \ell_1^*$ and $x\dbs \in \ell_1\dbs$.   Then $\bra{p^*}{x\dbs}_1 = \bra{Wx\dbs}{p^*}_0$.
\end{lemma}
\begin{proof}
For all $m \ge 1$, find $\delta_m$ such that $|\delta_m| = 1$ and $\delta_m\bra{e_m^*}{x\dbs}_1 = |\bra{e_m^*}{x\dbs}_1|$.   Let $n \ge 1$.   Then
\begin{align*}
\ts\sum_{m = 1}^n|\bra{e_m^*}{x\dbs}_1|
&= \ts\sum_{m = 1}^n\delta_m\bra{e_m^*}{x\dbs}_1
= \ts\Bra{\sum_{m = 1}^n\delta_me_m^*}{x\dbs}_1\\
&\le\ts\big\|\sum_{m = 1}^n\delta_me_m^*\big\|_\infty\|x\dbs\| = \ts\sup_{m = 1}^n|\delta_m|\|x\dbs\| = \|x\dbs\|.
\end{align*}
(a) now follows by letting $n \to \infty$.  It also follows that $\|Wx\dbs\|_1 \le \|x\dbs\|$.   Since this holds for all $x\dbs \in \ell_1\dbs$, $\|W\| \le 1$.   Now let $x \in \ell_1$.   Then, for all $m \ge 1$, $(W{\wh x})_m = \Bra{e_m^*}{{\wh x}}_1 = \bra{x}{e_m^*}_0 = x_m$, and so $W{\wh x} = x$, as required.   It follows from this that $\|W\| = 1$, which completes the proof of (b).
\par
Let $p^* = (p_m)_{m \ge 1} \in c_0$.   Since $p^* = \lim_{n \to \infty}\sum_{m = 1}^np_me_m^*$ in $\ell_\infty$,
\begin{align*}
\ts\bra{p^*}{x\dbs}_1
&= \ts\lim_{n \to \infty}\sum_{m = 1}^np_m\bra{e_m^*}{x\dbs}_1
= \ts\sum_{m = 1}^\infty p_m\bra{e_m^*}{x\dbs}_1\\
&= \ts\sum_{m = 1}^\infty\bra{e_m^*}{x\dbs}_1p_m = \bra{Wx\dbs}{p^*}_0.
\end{align*}
This completes the proof of (c).
\end{proof}
In what follows, we define $w\trs := \wh{e^*} - W^*e^* \in \ell_1\trs$.
%:   \Thm{Wthm}
\begin{theorem}\label{Wthm}
We have
%   \eqref{W1}
\begin{equation}\label{W1}
\|w\trs\| = 1,
\end{equation}
for all $x \in \ell_1$,
%   \eqref{W2}
\begin{equation}\label{W2}
\bra{\wh x}{w\trs}_2 = 0,
\end{equation}
and, for all $x^* = (x_m^*)_{m \ge 1} \in c$ and $x\dbs \in \ell_1\dbs$,
%   \eqref{W3}
\begin{equation}\label{W3}
\ts\bra{x^*}{x\dbs}_1 = \bra{Wx\dbs}{x^*}_0 + \bra{x\dbs}{w\trs}_2\lim_{n \to \infty}x_n^*.
\end{equation}
\end{theorem}
\begin{proof}
\par
Let $n \ge 1$.   Then
\begin{align*}
\ts\bra{e^*}{x\dbs}_1 - \sum_{m = 1}^n\bra{e_m^*}{x\dbs}_1
&= \ts\Bra{e^* - \sum_{m = 1}^ne_m^*}{x\dbs}_1\\
&\le \ts\big\|e^* - \sum_{m = 1}^ne_m^*\big\|_\infty\|x\dbs\| = \|x\dbs\|.
\end{align*}
Letting $n \to \infty$, $\bra{e^*}{x\dbs}_1 - \sum_{m = 1}^\infty\bra{e_m^*}{x\dbs}_1 \le \|x\dbs\|$.   Thus
%   \eqref{W4}
\begin{align}\label{W4}
\bra{x\dbs}{w\trs}_2
&= \Bra{x\dbs}{\wh{e^*}}_2 - \Bra{x\dbs}{W^* e^*}_2 = \bra{e^*}{x\dbs}_1 - \bra{Wx\dbs}{e^*}_0\\
&= \ts\bra{e^*}{x\dbs}_1 - \sum_{m = 1}^\infty\bra{e_m^*}{x\dbs}_1 \le \|x\dbs\|.\notag
\end{align}
Since this holds for all $x\dbs \in \ell_1\dbs$, $\|w\trs\| \le 1$.   On the other hand, if $x_0\dbs$ is as in \Lem{GTlem}, then \eqref{GT0} gives $\|x_0\dbs\| = 1$ and, from \eqref{R1} and the above, 
\begin{align*}
\bra{x_0\dbs}{w\trs}_2
&=\ts\bra{e^*}{x_0\dbs}_1 - \sum_{m = 1}^\infty\bra{e_m^*}{x_0\dbs}_1 = 1 - \sum_{m = 1}^\infty0 = 1, 
\end{align*}
which gives \eqref{W1}.   Now let $x \in \ell_1$.   Then, from \Lem{Wlem}(b),
\begin{align*}
\bra{\wh x}{w\trs}_2
&= \Bra{\wh x}{\wh{e^*}}_2 - \Bra{\wh x}{W^* e^*}_2 =  \bra{e^*}{\wh x}_1 - \bra{W\wh x}{e^*}_0\\
&= \bra{x}{e^*}_0 - \bra{x}{e^*}_0 = 0,
\end{align*}
which gives \eqref{W2}.   Finally, let $x^* = (x_m^*)_{m \ge 1} \in c$, $x\dbs \in \ell_1\dbs$, and write $\Lambda := \lim_{n \to \infty}x_n^*$.   From \Lem{Wlem}(c), with $p^* := x^* - \Lambda e^* \in c_0$,
\begin{equation*}
\bra{x^* - \Lambda e^*}{x\dbs}_1 = \bra{Wx\dbs}{x^* - \Lambda e^*}_0.
\end{equation*}
Thus     
\begin{align*}
\bra{x^*}{x\dbs}_1
&= \bra{Wx\dbs}{x^* - \Lambda e^*}_0 + \bra{\Lambda e^*}{x\dbs}_1\\
&= \bra{Wx\dbs}{x^*}_0 - \bra{Wx\dbs}{\Lambda e^*}_0 + \bra{\Lambda e^*}{x\dbs}_1,
\end{align*}
which gives \eqref{W3}.   This completes the proof of \Thm{Wthm}.
\end{proof}
\begin{remark}
For all $x\dbs \in \ell_1\dbs = \ell_\infty^*$, $\bra{x\dbs}{W^* e^*}_2 = \bra{Wx\dbs}{e^*}_0 =\break \sum_{m = 1}^\infty\bra{e_m^*}{x\dbs}_1 = \sum_{m = 1}^\infty\bra{x\dbs}{\wh{e_m^*}}_2$, so we could write $W^* e^* = \sum_{m = 1}^\infty\wh{e_m^*}$, with the understanding that the convergence is in the $w(\ell_1\trs,\ell_1\dbs)$ (weak$^*$) sense.   Thus, with this understanding, $w\trs = \wh{e^*} - \sum_{m = 1}^\infty\wh{e_m^*}$.   Note from \eqref{W1} that this does not imply that $w\trs = 0$. 
\par
Since $\wh{W\wh{x}} = \wh{x}$, the map $x\dbs \mapsto \wh{Wx\dbs}$ is a {\em linear retraction} from $\ell_1\dbs$ onto $\wh{\ell_1}$.
\end{remark}
%:   \Rem{BETArem}
\begin{remark}\label{BETArem}
It is known from standard results in point--set topology that the set $\NN$ of positive integers (considered as a discrete topological space) can be embedded as a dense open subspace of a compact Hausdorff space, $\beta \NN$ (the {\em Stone-\u Cech compactification} of $\NN$), such that, for all $x^* \in \ell_1^* = {\ell_\infty}$, there exists a unique element $\beta x^*$ of $C(\beta \NN)$ (the set  of continuous functions on $\beta \NN$) extending $x^*$.   (The fact that $\NN$ is open in $\beta\NN$ is a consequence of the result proved in \cite[XI.8.3, pp. 245--246]{DUG} that any locally compact completely regular space is open in any compactification.)    Obviously $\beta e^* = 1$.
\par
For all $m \ge 1$, $\{m\}$ is open (in $\NN$ and hence) in $\beta\NN$, and so, if $f_m\colon \beta\NN \to \RR$ is defined by $f_m(m) := 1$ and $f_m := 0$ on $\beta \NN \setminus \{m\}$ then $f_m = \beta e_m^* \in C(\beta \NN)$.   It follows from the Riesz representation theorem (see, for instance, \cite[Theorem 6.19, pp. 130--132]{RUDIN} for a considerably more general result) that $\ell_1\dbs = \ell_\infty^*$ can be identified with the set $\cal M(\beta \NN)$ of (signed) Radon measures on $\beta \NN$.   If $x\dbs \in \ell_1\dbs = \ell_\infty^*$ and $\mu \in \cal M(\beta \NN)$ represents $x\dbs$ then $\bra{e_m^*}{x\dbs}_1 = \int f_md\mu = \mu(\{m\})$, and so $Wx\dbs = \big(\mu(\{m\})\big)_{m \ge 1} \in \ell_1$.   Furthermore, $\bra{e^*}{x\dbs}_1 = \int 1d\mu = \mu(\beta \NN$).   Thus, from \eqref{W4} and standard measure-theoretic arguments,
%:   \eqref{W5} 
\begin{equation}\label{W5}
\bra{x\dbs}{w\trs}_2
= \ts\mu(\beta \NN) - \sum_{m = 1}^\infty\mu(\{m\}) = \mu(\beta \NN \setminus \NN).
\end{equation}
This discussion will be continued in \Rem{GTMrem}. 
\end{remark}
%:   \Sec{FINERsec}
\section{$G^*$}\label{FINERsec}
\Thm{TRZthm} below extends the results proved in \eqref{GT1} and \eqref{GT2} for a particular element $x_0\dbs$ of $\ell_1\dbs$ to a general element $x\dbs$ of $\ell_1\dbs$. 
%   \Thm{TRZthm}
\begin{theorem}\label{TRZthm}
Let $x\dbs \in \ell_1\dbs$.   Then
%   \eqref{TRZ2}
\begin{equation}\label{TRZ2}
G^* x\dbs = -GWx\dbs - \Bra{x\dbs}{w\trs}_2e^* \in \ell_1^*
\end{equation}
and
%   \eqref{TRZ1}
\begin{equation}\label{TRZ1}
\bra{G^* x\dbs}{x\dbs}_1 = -{\bra{x\dbs}{w\trs}}_2^2.
\end{equation}
\end{theorem}
\begin{proof}
Let $x \in \ell_1$.   Setting $x^* = Gx$ in \eqref{W3}, writing $\alpha$ for $\bra{x\dbs}{w\trs}_2$ to simplify the expressions, and using \eqref{SKEW1} (with  $w = Wx\dbs$) and \eqref{H1}, 
\begin{align*}
\ts\bra{x}{G^* x\dbs}_0
&= \ts\bra{Gx}{x\dbs}_1 = \bra{Wx\dbs}{Gx}_0 + \alpha\lim_{n \to \infty}(Gx)_n\\
&= -\bra{x}{GWx\dbs}_0 - \alpha\bra{x}{e^*}_0.
\end{align*}  
Since this holds for all $x \in \ell_1$, this completes the proof of \eqref{TRZ2}.   From \eqref{TRZ2}, the definition of $G^*$, \eqref{TRZ2} again, \eqref{SKEW2} (with $x = Wx\dbs$), the definition of $W^*$, 
 and the definition of $w\trs$ (in sequence),
\begin{align*}
\bra{G^* x\dbs}{x\dbs}_1
&=-\bra{GWx\dbs}{x\dbs}_1 - \alpha\bra{e^*}{x\dbs}_1\\
&=-\bra{Wx\dbs}{G^* x\dbs}_0 - \alpha\bra{e^*}{x\dbs}_1\\
&=\bra{Wx\dbs}{GWx\dbs}_0 + \alpha\bra{Wx\dbs}{e^*}_0 - \alpha\bra{e^*}{x\dbs}_1\\
&=\alpha\bra{Wx\dbs}{e^*}_0 - \alpha\bra{e^*}{x\dbs}_1\\
&= \alpha\big[\bra{x\dbs}{W^* e^*}_2 - \bra{x\dbs}{\wh{e^*}}_2\big] = -\alpha^2.
\end{align*}
This gives \eqref{TRZ1}, and completes the proof of \Thm{TRZthm}.
\end{proof}
\begin{remark}
There is an analysis of $G^*$ and $\bra{G^* x\dbs}{x\dbs}_1$ in \cite[Example 14.2.2, pp.\ 161--162]{HEINZ} that is, on the surface, different from the one presented in \Thm{TRZthm} above.   A linear operator $T\colon\ \ell_1 \to \ell_\infty$ is defined by
\begin{equation*}
\all\ x \in \ell_1,\ (Tx)_m = \ts x_m + 2\sum_{n > m}x_n.
\end{equation*}
If $n$ is odd, we define $y^{(n)} \in \ell_1$ by $y^{(n)} := (2,-2,\dots,2,-2,1,0,0,\dots)$, where the ``$1$'' is in the $n$th place.   Then
\begin{equation*}
T\big(y^{(n)}\big) = (2,-2,\dots,2,-2,1,0,0,\dots) + 2\big(-1,1,\dots,-1,1,0,0,0,\dots\big) = e_n^*.
\end{equation*}
Similarly, if $n$ is even, we define $y^{(n)} \in \ell_1$ by $y^{(n)} := (-2,2,\dots,2,-2,1,0,0,\dots)$.   Again, $T\big(y^{(n)}\big) = e_n^*$.   The analysis in \cite{HEINZ} rests on the assumption (see\break \cite[Eqn. (1), pp.\ 159]{HEINZ}) that, for all $x\dbs \in \ell_1\dbs$,
\begin{equation*}
\exs\ x \in \ell_1\ \st,\ \all\ x^* \in R(T), \bra{x^*}{\wh x}_1 = \bra{x^*}{x\dbs}_1.
\end{equation*}
In particular, the argument above implies that, for all $n \ge 1$, $\bra{e_n^*}{\wh x}_1 = \bra{e_n^*}{x\dbs}_1$.   Thus, for all $n \ge 1$, $x_n = \bra{e_n^*}{x\dbs}_1$.   Consequently, $x = Wx\dbs$, and the formulae for $G^*$ and $\bra{G^* x\dbs}{x\dbs}_1$ given in \cite{HEINZ} reduce to the more\break explicit ones given in \Thm{TRZthm} above.    
\end{remark}
%:   \Rem{GTMrem}
\begin{remark}\label{GTMrem}
This is a continuation of \Rem{BETArem}. A comparison of \eqref{W5} and \eqref{TRZ1} leads to the conclusion that $\bra{G^* x\dbs}{x\dbs}_1 = -\mu(\beta \NN \setminus \NN)^2$.   This is exactly the formula obtained in Gossez, \cite[Example, p.\ 89]{GOSSEZRANGE}. 
\end{remark}
%:bibliography

\end{document}